\documentclass[12pt, twoside, leqno]{article}
\usepackage{amsmath,amsthm}
\usepackage{amssymb,latexsym}
\usepackage{enumerate}

\usepackage{latexsym}
\usepackage{amsfonts}
\usepackage[english]{babel}
\usepackage[latin1]{inputenc}
\usepackage{babel}
\usepackage{eufrak}

\newtheorem{thm}{Theorem}[section]

\newtheorem{lem}[thm]{Lemma}
\newtheorem{prob}[thm]{Problem}
\newtheorem{clm}{Claim}

\newtheorem{mainthm}[thm]{Main Theorem}

\theoremstyle{definition}
\newtheorem{defin}[thm]{Definition}

\newcommand{\N}{\mathbb{N}}

\title{Suprema of continuous functions on connected spaces}
\author{André Santoleri Villa Barbeiro \thanks{Supported by a grant from CAPES,
number 1328310.} \thanks{E-mail: andresan@ime.usp.br}
\and Rogério Augusto dos Santos Fajardo \thanks{E-mail: fajardo@ime.usp.br}
\medskip
\\ \normalsize Instituto de Matemática e Estatística
\\ \normalsize Universidade de São Paulo
\\ \normalsize Rua do Matão, 1010 -- São Paulo -- SP -- Brazil
\\ \normalsize CEP 05508-900}
\date{}
\begin{document}
\maketitle

\renewcommand{\thefootnote}{}



\renewcommand{\thefootnote}{\arabic{footnote}}

\begin{abstract} Let $K$ be a compact Hausdorff space and let $(f_n)_{n\in \N}$
be a pairwise disjoint sequence of continuous functions from $K$ into $[0,1]$. 
We say that a compact space $L$ \emph{adds supremum} of $(f_n)_{n\in \N}$ in $K$
if there exists a continuous surjection $\pi:L\longrightarrow K$ such that
there exists $sup\{f_n\circ\pi:n\in \N\}$ in $C(L)$. Moreover, we expect that
$L$ preserves suprema of disjoint continuous functions which already existed in $C(K)$. 
Namely, if $sup\{g_n:n\in \N\}$ exists in $C(K)$, we must have 
$sup\{g_n\circ\pi:n\in \N\}$ in $C(L)$.

This paper studies the preservation of connectedness in extensions by continuous functions 
--  a technique developed by Piotr Koszmider to add suprema of continuous functions on 
Hausdorff connected compact spaces -- proving the following results:

\begin{description}
\item[(1)] If $K$ is a metrizable and locally connected compactum, then any extension of $K$
by continuous functions is connected (but it may be not locally connected).
\item[(2)] There exists a disconnected extension of a metrizable connected compactum $K$. 
\item[(3)] For any metrizable compactum $K$ there exists a disconnected $L$ which is
obtained from $K$ by finitely many extensions by continuous functions.
\end{description}
 
\end{abstract}

\section{Introduction}\label{sec:intro}

In \cite{Ko1}, Koszmider constructed a totally disconnected compact Hausdorff 
space $K$ such that every operator on $C(K)$ is a multiplication by continuous function
plus a weakly compact operator. The space $K$ was obtained as a Stone space of a
Boolean algebra constructed by transfinite induction, where each $\mathcal{A}_{\alpha+1}$
is the Boolean algebra generated by $\mathcal{A}_\alpha$ and the supremum of an antichain
in $\mathcal{A}_\alpha$. In the limit step is taken the union.

The same paper also proved the existence of an indecomposable $C(K)$.
The construction is similar to the previous one, except that $K$ must be connected.
Hence, $K$ cannot be the Stone space of an Boolean algebra. 
The key of the construction of $K$ was to replace the supremum of elements of
the Boolean algebra -- which corresponds, in its Stone space, to the supremum of the
characteristic functions of those elements -- by the supremum of disjoint continuous 
functions. At this point, Koszmider developed a new technique which is the subject of
this paper.

Before introducing the particular construction of Koszmider, we present
an overview of the problem. Let $K$ be a compact Hausdorff space.
Define $C(K)$ the Banach space of all continuous real functions on $K$, normed by
the supremum, and $C_1(K)$ the set of all continuous functions from $K$ into $[0,1]$.
We say that $f,g\in C(K)$ are \emph{disjoint} if $f\cdot g=0$.

Let $(f_n)_{n\in\N}$ be a pairwise disjoint sequence in $C_1(K)$.
Let $L$ be another compact Hausdorff space and 
$\pi$ a continuous surjection from $L$ to $K$. We say that
$(L,\pi)$ \emph{adds the supremum of} $(f_n)_{n\in\N}$ if
$sup\{f_n\circ\pi:n\in\N\}$ exists in $C(L)$. Moreover, we say that $(L,\pi)$
\emph{preserves suprema} if $sup\{g_n\circ\pi:n\in\N\}$ exists in $C(L)$,
whenever $(g_n)_{n\in\N}$ is a pairwise disjoint sequence in $C_1(K)$ 
which has supremum in $C(K)$.
  
Now we will see how it works in the case of Boolean algebras.
Let $\mathcal{A}$ be a Boolean algebra and $(a_n)_{n\in\N}$ a pairwise disjoint sequence
in $\mathcal{A}$. Let $b$ be the supremum of $(a_n)_{n\in\N}$ in the completation
of $\mathcal{A}$. Take $\mathcal{B}$ the algebra generated by $\mathcal{A}\cup\{b\}$.

Let $S(\mathcal{A})$ and $S(\mathcal{B})$ be the Stone spaces of $\mathcal{A}$ and
$\mathcal{B}$, respectively. Let $\pi$ be the standard projection from $S(\mathcal{B})$
onto $S(\mathcal{A})$, given by $\pi(u)=u\cap \mathcal{A}$, 
whenever $u$ is an ultrafilter in $\mathcal{B}$.
 
For any $a\in\mathcal{B}$, we denote by $[a]_\mathcal{B}$ the clopen set of $S(\mathcal{B})$ 
consisting of all ultrafilters on $\mathcal{B}$ which contain $a$. If $a\in\mathcal{A}$,
the notation $[a]_\mathcal{A}$ means the set of all ultrafilters on $\mathcal{A}$ 
which contain $a$.
It is easy to see that $\pi[[a]_\mathcal{B}]=[a]_\mathcal{A}$.

In $C(S(\mathcal{B}))$ take $\chi_{[b]_\mathcal{B}}$ the characteristic function of
$[b]_\mathcal{B}$, which is continuous by the fact that $[b]_\mathcal{B}$ is a clopen 
set. Since $b$ is the supremum of $(a_n)_{n\in \N}$, $\chi_{[b]_\mathcal{B}}$ is
clearly the supremum of $\{\chi_{[a_n]_\mathcal{B}}:n\in\N\}$.  For each $n\in\N$ we
have $\chi_{[a_n]_\mathcal{B}}=\chi_{[a_n]_\mathcal{A}}\circ\pi$. 
Then we conclude that $(S(\mathcal{B}),\pi)$ adds the supremum of 
$(\chi_{[a_n]_\mathcal{B}})_{n\in \N}$.

If $(a'_n)_{n\in\N}$ has supremum $a$ in $\mathcal{A}$, then
$\chi_{[a]_\mathcal{B}}$ is the supremum of $(\chi_{[a'_n]_\mathcal{B}})_{n\in\N}$
in $C(S(\mathcal{B}))$. Hence, $S(\mathcal{B})$ preserves suprema of functions of
this kind, i.e., which are characteristic functions of basic clopen sets. 
For the general case, it follows from Lemma 4.1 of \cite{Ko1} that
$(C(S(\mathcal{B})),\pi)$ preserves suprema.

Adding suprema of elements of a Boolean algebra is much easier than adding suprema in $C(K)$. 
Therefore, the best approach to add suprema
in $C(K)$ for $K$ totally disconnected is via Boolean algebras.
If $K$ is connected, Koszmider, in \cite{Ko1}, introduced the definition of
\emph{extension by continuous functions} (see definition~\ref{defin:extensao})
to obtain a compactum $L\subset K\times [0,1]$ which adds supremum of a given pairwise
disjoint sequence in $(f_n)_{n\in\N}$ in $C_1(K)$. 
The function $\pi$ is the standard projection on $K$.
We use the notation $K((f_n)_{n\in\N})$ for the extension of $K$ by $(f_n)_{n\in\N}$. 

Extensions by continuous functions were applied in several 
counterexamples in the theory of Banach spaces of the form $C(K)$, as it is shown in 
\cite{Ko1}, \cite{Fa1} and \cite{Fa2}. Alternative ways of adding suprema
in connected spaces were developed in \cite{Pl} -- using Wallman representation,
which generalizes the Stone representation for connected lattices --
and \cite{Ko3} -- using ranges of Stone spaces of Boolean algebras.

The main difficulty in the constructions that use extensions by continuous
functions is to assure connectedness. For this, 
we need, eventually, to go to a subsequence (as in \cite{Ko1}) or to modify
the functions (as in \cite{Fa2}). 

Along this paper, we will frequently work with metrizable, connected and compact spaces.
Then, it is convenient to use the terminology below:

\begin{defin}\label{defin:continuum} A topological space $K$ is called \emph{continuum}
if it is metrizable, compact and connected.
\end{defin}

The first aim of this paper is investigate the following question:

\begin{prob}\label{prob1} Suppose that $K$ is a continuum and $L$ is an extension
of $K$ by a pairwise disjoint sequence $(f_n)_{n\in\N}$ in $C_1(K)$. Is $L$ connected?
\end{prob}

In Theorem~\ref{thm:disconnected1} we construct a continuum $K$ and functions
$(f_n)_{n\in\N}$ such that $K((f_n)_{n\in\N})$ is disconnected.
Theorem~\ref{thm:disconnected2} adapts that construction in a way that
$K$ itself is an extension of $[0,1]$. 

Once we have answered negatively to Problem~\ref{prob1}, we may ask what happens if
we replace connectedness by some stronger property. The following question arises:

\begin{prob}\label{prob2} Is there some non-void property $P$ on the class
of continua such that, whenever $K$ satisfies $P$ and 
$(f_n)_{n\in\N}$ is a pairwise disjoint sequence in $C_1(K)$,
we have $K((f_n)_{n\in\N})$ satisfying $P$?
\end{prob}

We could try the property $P$ as being locally connectedness.
In fact, Theorem~\ref{thm:locally} proves that any extension of a metrizable
locally connected $K$ is connected. But the extension may be not locally connected.

Problem~\ref{prob2} is also answered negatively. Theorem~\ref{thm:disconnected3}
proves that, starting from any metrizable continuum $K_0$, there exists a
sequence $(K_i)_{0\leq i\leq 3}$ of compact Hausdorff spaces such that each $K_i$
is an extension of $K_{i-1}$ by continuous functions and $K_3$ is disconnected.

Assuming the Continuum Hypothesis, problems~\ref{prob1} and~\ref{prob2} can be solved indirectly (with the additional 
condition that $P$ is preserved in inverse limits, in Problem~\ref{prob2}) 
using the construction in \cite{Ko1} and a theorem which can be found in \cite{Me}. 
If one of these problems had affirmative answer we could construct, as in \cite{Ko1}, 
a compact connected Hausdorff space such that every
disjoint sequence in $C_1(K)$ has supremum. By \cite{Me} it implies that K is quasi-Stonean\footnote{A compact
Hausdorff space is quasi-Stonean if the closure of every open $F_\sigma$ is open.} and therefore disconnected. Nevertheless, we present visual examples which
show how connectedness is lost in successive extensions.

Although the main theorem of this paper may has no direct application to the theory of Banach spaces with few operators, it is important
to the field, since it disproves a conjecture which would simplify several constructions.

All topological spaces in this paper are Hausdorff.

\section{Extensions by continuous functions}

In this section we state the main definitions and lemmas about extensions
by continuous functions, based on \cite{Ko1}. If $f$ is a real function on a compact $K$,
we denote by $supp(f)$ the closure of $\{x\in K:f(x)\neq 0\}$ in $K$.

\begin{defin}\label{defin:opendense}
Let $K$ be a compact space and let $(f_n)_{n\in\N }$ be a pairwise disjoint sequence in
$C_1(K)$. We define 
$$D((f_n)_{n\in\N })=\bigcup\{U:U\mbox{ is open and }\{n:U\cap supp(f_n)\neq\emptyset\}
\mbox{ is finite}\}.$$ 
\end{defin}

\begin{lem}[\cite{Ko1}, 4.1]\label{lem:opendense}
Let $K$ be a compact space and let $(f_n)_{n\in\N }$ be a pairwise disjoint sequence in
$C_1(K)$. Then:
\begin{enumerate}[\upshape (i)] 
\item $f\in C(K)$ is $sup\{f_n:n\in\N \}$ in the lattice of $C(K)$ if and only if
		$$\{x\in K:\Sigma_{n\in\N }f_n(x)\neq f(x)\}$$
		is nowhere dense in $K$;
\item $D((f_n)_{n\in\N })$ is an open dense set in $K$ and
		$\Sigma_{n\in\N }f_n$ is continuous on $D((f_n)_{n\in\N })$.
\end{enumerate}
\end{lem}

\begin{defin}[\cite{Ko1}, 4.2]\label{defin:extensao}
Suppose that $K$ is compact, $L\subset K\times [0,1]$
and $(f_n)_{n\in\N }$ is a pairwise disjoint sequence of continuous
functions from $K$ into $[0,1]$.
We say that $L$ is an \emph{extension of} $K$ \emph{by} $(f_n)_{n\in\N }$, which we will
denote by $K((f_n)_{n\in\N })$, if $L$ is the closure of the graph of
$\Sigma_{n\in\N }f_n|D((f_n)_{n\in\N })$. 
Moreover, we say that $L$ is a \emph{strong extension} if it contains the graph
of $\Sigma_{n\in\N }f_n$.
\end{defin}



\begin{lem}[\cite{Ko1}, 4.3 and 4.4]\label{lem:supremo}
Let $(f_n)_{n\in \N}$ be a pairwise disjoint sequence in $C_1(K)$.
Take $L=K((f_n)_{n\in \N})$ and $\pi:L\longrightarrow K$ the standard projection. 
Then $(L,\pi)$ adds the supremum of $(f_n)_{n\in \N}$ and preserves suprema.
Moreover, if $L$ is a strong extension and $K$ is connected, then $L$ is connected.
\end{lem}

In \cite{Ko1} it is proven that the projection $\pi$ preserves nowhere dense sets,
i.e., $\pi^{-1}[M]$ is nowhere dense in $L$, wherever $M$ is nowhere dense in $K$.
Hence, preservation of suprema follows from lemma~\ref{lem:opendense}.

The supremum of $(f_n\circ\pi)_{n\in \N}$ in $L$ is the projection in
the second coordinate, i.e., the function $f$ defined as $f(x,t)=t$, for
$(x,t)\in L\subset K\times [0,1]$.

Next lemma is a simplified version of lemma 4.5 of \cite{Ko1}. It shows that
the extension is strong for many subsequences of $(f_n)_{n\in\N}$.

\begin{lem}[\cite{Ko1}, 4.5]\label{lem:strongextension}
 Let $K$ be a compact space of topological weight
$\kappa<2^\omega$. Suppose that $(f_n)_{n\in\N}$ is a pairwise disjoint sequence
in $C_1(K)$ and $(N_\xi)_{\xi<\omega}$ is a family of infinite subsets of $\N$ such
that $N_\xi\cap N_{\xi'}$ is finite, for all $\xi\neq\xi'$. 
Then there exists $A\subset 2^\omega$ of cardinality not bigger than $\kappa$
such that, for every $\xi\in 2^\omega\smallsetminus A$ and $b\subset N_\xi$,
the extension of $K$ by $(f_n)_{n\in b}$ is strong.
\end{lem}

In \cite{Fa2}, lemmas 3.6 and 3.8 provide another way of obtaining connectedness
in the extension by modifying slightly the functions $f_n$.

\begin{lem}[\cite{Fa2}, 3.6 and 3.8]\label{lem:modifyfunctions}
Let $K$ be a continuum. 
Suppose that $(\varepsilon_n)_{n\in \N }$ is a sequence 
of positive real numbers, $(f_n)_{n\in \N }$ is a pairwise disjoint sequence in 
$C_1(K)$, $(\mu_n)_{n\in \N}$ is a sequence of regular
measures on $K$ and $(x_n)_{n\in \N}$ is a sequence of points in $K$ such that $f_n(x_n)=1$. 
Then there exists a sequence $(f'_n)_{n\in\N}$ in $C_1(K)$ such that
\begin{enumerate}[\upshape (i)]
\item For every $n\in  \N $, $supp(f'_n)\subset supp(f_n)$, 
		$f'_n(x_n)=1$ and \newline $\int |f'_n-f_n|\,d|\mu_n|<\varepsilon_n$;
\item $K((f'_n)_{n\in\N})$ is connected.   
\end{enumerate}  
\end{lem}

In spite of all this effort to ensure connectedness in the extension,
none of the mentioned papers proves that extensions do not preserve
connectedness, in general. That is the aim of Section \ref{sec:disconnected}.

\section{Connectedness of some extensions by continuous functions}\label{sec:connected}

We recall that a space is said locally connected if it contains a basis of
connected open sets.

\begin{thm}\label{thm:locally} Let $K$ be a metrizable and locally connected 
compact space. Then every extension by continuous functions of $K$ is strong.
\end{thm}

\begin{proof} Let $K$ be as in the hypothesis and let $(f_n)_{n\in\N}$ be a pairwise
disjoint sequence in $C_1(K)$. Take $L=K((f_n)_{n\in\N})$. We will prove that
$L$ is a strong extension and, for this, it is enough to show that $(x,0)\in L$,
for every $x\in K\smallsetminus D((f_n)_{n\in\N})$.
Since $L$ is metrizable and compact, we will prove that $(x,0)\in L$ constructing
a sequence $(x_n,r_n)$ in $L$ converging to $(x,0)$.

Let $(V_n)_{n\in\N}$ be a local basis for $x$, where each $V_n$ is an open connected set.
Fix $n\in\N$. By definition~\ref{defin:opendense} there exists $k_n\in\N$ such that 
$V_n\cap supp(f_{k_n})\neq\emptyset$. Hence, there exists $y_n\in V_n$ such that
$f_{k_n}(y_n)>0$. Since $V_n$ is connected and $f_{k_n}$ is continuous, there exists
$x_n\in V_n$ such that $0<f_{k_n}(x_n)<\frac{1}{n}$. Take $r_n=f_{k_n}(x_n)$.
Clearly we have $x_n\in D((f_n)_{n\in\N})$ and $\Sigma_{i\in\N} f_i(x_n)=r_n$.
Therefore, $(x_n,r_n)\in L$. Clearly $(x_n,r_n)\rightarrow (x,0)$, concluding
the theorem.
\end{proof}

\section{Disconnected extensions of continua}\label{sec:disconnected}

\begin{thm}\label{thm:disconnected1} 
There exist a continuum $K$ and a pairwise disjoint sequence
$(f_n)_{n\in\N}$ in $C_1(K)$ such that $K((f_n)_{n\in\N})$ is disconnected.
\end{thm}

\begin{proof} For each $n\in\N$, define $g_n:[0,1]\longrightarrow [0,1]$ as
$$g_n(x)=\left\{\begin{array}{ll}
         sin(\pi(2^{n+1}x-1)), & \mbox{if\ }x\in[\frac{1}{2^{n+1}},\frac{1}{2^n}] \\
         0, &\mbox{otherwise}.
        \end{array}
		\right. $$
		
Note that $g_n(\frac{1}{2^{n+1}})=g_n(\frac{1}{2^n})=0$ and $g_n(\frac{3}{2^{n+2}})=1$.
Moreover, $g_n$ is strictly increasing on $[\frac{1}{2^{n+1}},\frac{3}{2^{n+2}}]$
and strictly decreasing on $[\frac{3}{2^{n+2}},\frac{1}{2^n}]$.

For brevity, we will call $a_n=\frac{1}{2^{n+1}}$ and $b_n=\frac{3}{2^{n+2}}$, for all $n\in\N$.

Let $K_0$ be the extension of $[0,1]$ by $(g_n)_{n\in\N}$.
Define $K=K_0\cup(\{0\}\times[1,2]$) and, for each $n\in\N$, consider the function
$f_n:K\longrightarrow [0,1]$ given by
$$f_n(x,t)=\left\{\begin{array}{ll}
         t, & \mbox{if\ }x\in(a_{n+1},a_n)\\
         0, &\mbox{otherwise}.
        \end{array}
		\right. $$
		
By Theorem~\ref{thm:locally}, $K_0$ is connected. Since $K_0$ is closed in 
$K\times[0,1]$ and $((b_n,1))_{n\in\N}$ is a sequence in $K_0$ converging to
$(0,1)$, we have $(0,1)\in K_0\cap(\{0\}\times [1,2])$. Since both are connected,
we conclude that $K$ is connected.

Using the continuity of $g_n$, it is easy to verify that each $f_n$ is continuous.
The sequence $(f_n)_{n\in\N}$ is clearly pairwise disjoint.

Let $L=K((f_n)_{n\in\N})$. We note that
$$\{0\}\times (1,2]\subset D((f_n)_{n\in\N}),$$
since it is itself an open set of $K$ which intersects none of the
supports of $f_n$. Therefore, $\{0\}\times (1,2]\times\{0\}\subset L$ and,
since $L$ is closed, we have $\{0\}\times [1,2]\times\{0\}\subset L$. 

Clearly, $\{0\}\times [1,2]\times\{0\}$ is closed in $L$. To conclude the
theorem it is sufficient to prove that this set is open in $L$.

Let $(z_n)_{n\in\N}$ be a sequence in 
$L\smallsetminus\{0\}\times [1,2]\times\{0\}$ converging to $z\in L$.
We have to prove that $z\notin\{0\}\times [1,2]\times\{0\}$.

We may assume that $z_n\in Graph(\Sigma_{n\in\N} f_n|D((f_n)_{n\in\N}))$, 
since it is dense in $L$.
 
Note that, if $(0,t)\in K$, for $0<t\leq 1$, any open neighborhood of $(0,t)$
intersects all but finitely many $supp(f_n)$. This follows from the fact
that each $g_n$ is surjective on $[0,1]$, and the supports of $g_n$ converge
to $0$.

Hence, we have $(0,t)\notin D((f_n)_{n\in\N})$, for $0<t\leq 1$.
Therefore, $z_n\notin \{0\}\times (0,1]\times[0,1]$.

Let $z_n=(x_n,y_n,t_n)$ and $z=(x',y',t')$. 
Suppose that $z\in \{0\}\times [1,2]\times\{0\}$. I.e., $z=(0,y',0)$,
where $1\leq y'\leq 2$. Going to a subsequence, we may assume that $y_n>\frac{1}{2}$, 
for all $n$. By the above remark, we may assume that $x_n>0$,
for all $n$. But it is easy to verify that this implies that 
$$x_n\in supp(g_{k_n}),$$
for some $k_n\in\N$.  So, for all $n\in\N$,
$$(x_n,y_n)\in supp(f_{k_n}).$$
Therefore, $t_n=y_n$, for all n. But $y_n>\frac{1}{2}$, contradicting that
$(x_n,y_n,t_n)$ converges to $(0,y',0)$.

We showed that $\{0\}\times [1,2]\times\{0\}$ is a clopen set of $L$,
proving that $L$ is disconnected. 
\end{proof}

\begin{thm}\label{thm:disconnected2}  
There exist a continuum $K$ and a pairwise disjoint sequence
$(f_n)_{n\in\N}$ in $C_1(K)$ such that $K$ is an extension of $[0,1]$ by 
continuous functions and $K((f_n)_{n\in\N})$ is disconnected.
\end{thm}

\begin{proof} Without loss of generality, we will replace $[0,1]$
in the theorem by $[-1,1]$, to keep the notation of theorem~\ref{thm:disconnected1}. 
Let $g_n$, $a_n$, $b_n$ be as in the proof of theorem~\ref{thm:disconnected1},
extending $g_n$ to $[-1,1]$ defining $g_n(x)=0$, for $x\leq 0$. 
Define $K$ the extension of $[0,1]$ by $(g_n)_{n\in\N}$ and
$f_n:K\longrightarrow[0,1]$, for $n\in\N$, as
$$f_n(x,t)=\left\{\begin{array}{ll}
         1-t, & \mbox{if\ }x\in[b_n,b_{n+1}] \\
         0, &\mbox{otherwise}.
        \end{array}
		\right. $$

For each $n\in\N$ we will prove continuity of $f_n$. 
Let $(x_i,t_i)_{i\in\N}$ be a sequence converging to $(x,t)$ in $K$. 
We need to prove that $f_n(x_i,t_i)$ converges to $f_n(x,t)$.

If $x\in [b_n,a_{n+1})$, we have $x\in supp(g_n)\cap D((g_i)_{i\in\N})$ and,
hence, $t=g_n(x)$. By continuity of $g_n$ we have $f_n(x_i,t_i)=(x_i,g_n(x_i),1-g_n(x_i))$
converging to $f_n(x,t)=(x,g_n(x),1-g_n(x))$.

Analogously, if $x\in (a_{n+1},b_{n+1}]$ we have $t=g_{n+1}(x)$ and
$f_n(x_i,t_i)$ converges to $f_n(x,t)$.

If $x=a_{n+1}$, both $g_n(x_i)$ and $g_{n+1}(x_i)$ converge to $0$, for $i\in\N$.
So $f_n(x_i,t_i)$ converges to $1=f_n(x,t)$.

If $x\leq b_n$ or $x\geq b_{n+1}$ it is easy to verify that $f_n(x_i,t_i)$
converges to $0$, which is equal to $f_n(x,t)$.

Let $L$ be the extension of $K$ by $(f_n)_{n\in\N}$. 
We will prove that $L$ is disconnected. For this, we have to show that
$[-1,0]\times\{0\}\times\{0\}$ -- which is clearly a closed subset of $L$ --
is open in $L$.

The proof is analogous to Theorem~\ref{thm:disconnected1}.
Let $z_n=(x_n,y_n,t_n)$ be a sequence in $L\smallsetminus [-1,0]\times\{0\}\times\{0\}$
converging to $z=(x,y,t)\in L$. We have to prove that $z\notin L$.

We may assume that $z_n\in Graph(\Sigma_{n\in\N} f_n|D((f_n)_{n\in\N}))$, since it is dense in L.
In particular, $(x_n,y_n)\in D((f_n)_{n\in\N})$.

We notice that $\{0\}\times (0,1)$ is included in $K$, but it is
disjoint from $D((f_n)_{n\in\N})$. So we may assume that $x_n\neq 0$, for all n. We note that
$$\pi^{-1}_{L,[-1,1]}([-1,0))=[-1,0)\times\{0\}\times\{0\},$$
since $[-1,0)\subset D((g_n)_{n\in\N})$ and
$[-1,0)\times\{0\}\subset D((f_n)_{n\in\N})$.
So we also may assume that $x_n>0$, for all n.

Suppose $z\in [-1,0]\times\{0\}\times\{0\}$. In particular,
$z=(x,0,0)$, for some $x$.
We may assume that $y_n<\frac{1}{2}$, for all n.
Since $x_n>0$, there exists $i$ such that $x_n\in [b_i,b_{i+1}]$.
Hence, $$t_n=f_i(x_n,y_n)=1-y_n>\frac{1}{2},$$
which contradicts that $t_n$ converges to $0$.
\end{proof}


Before we state our main theorem, we need the following lemma:

\begin{lem}\label{lem:restriction} Let $K$ be a compactum and let $(f_n)_{n\in \N}$ 
be a pairwise disjoint sequence in $C_1(K)$. Suppose there exists a closed $L\subseteq K$ 
such that $supp(f_n)\subset L$, for all $n\in\N$. Then, 
$$K((f_n)_{n\in\N})=L((f_n|L)_{n\in\N})\bigcup (\overline{K\smallsetminus L})\times\{0\}.$$
\end{lem}

\begin{proof}  Let denote by $D$ the set $D((f_n)_{n\in\N})$. We first note that
$K\smallsetminus L$ is an open set in $K$ where $f_n$ is null for all $n$.
So we have $K\smallsetminus L\subset D$.
Since $supp(f_n)\subset L$, we have $V\cap supp(f_n)\neq\emptyset$ iff
$(V\cap L)\cap supp(f_n|L)\neq\emptyset$, whenever $n\in\N$ and $V$ is open in $K$.
Hence, $$D((f_n|L)_{n\in\N})=D\cap L.$$
Then we conclude that $$L((f_n|L)_{n\in\N})=\overline{Graph(\Sigma_{n\in\N} f_n|D\cap L)}.$$
This is sufficient to prove the lemma, since we have the following equalities:
$$K((f_n)_{n\in\N})=\overline{Graph(\Sigma_{n\in\N} f_n|D)}=
\overline{Graph(\Sigma_{n\in\N} f_n|D\cap L)\cup (K\smallsetminus L)\times\{0\}}.$$

\end{proof}

\begin{mainthm}\label{thm:disconnected3}
Let $K_0$ be any continuum. There exist compacta $K_1$, $K_2$
and $K_3$ such that $K_i$ is an extension of $K_{i-1}$ by continuous functions, 
for $1\leq i\leq 3$, and $K_3$ is disconnected.  
\end{mainthm}

\begin{proof} Using metrizability of $K_0$, fix $(\bar{x}_n)_{n\in\N}$ a convergent sequence
in $K_0$ and let $\bar{x}$ be its limit. Let $(U_n)_{n\in\N}$ be a pairwise disjoint sequence
of open sets such that $\bar{x}_n\in U_n$ and $diam(U_n)\leq \frac{1}{n}$
(i.e., for a fixed metric $d$ on $K_0$, we have $d(x,y)\leq \frac{1}{n}$, 
for all $x,y\in U_n$). Using Urysohn's Lemma and normality, 
we find, for each $n$, a function $e_n:K_0\longrightarrow [0,1]$ whose support is included 
in $U_n$ and such that $e_n(\bar{x}_n)=1$. By connectedness of $K_0$, $e_n[K_0]=[0,1]$.

Let $K_1=K_0((e_n)_{n\in\N})$.

\begin{clm}\label{clm:nowheredense1} We may assume that $e_n^{-1}[\{\frac{1}{2}\}]$ is
nowhere dense in $K_0$.
\end{clm}

Since compact metrizable spaces satisfy the countable chain condition, 
there exists $r\in [0,1]$
such that $e_n^{-1}[\{r\}]$ is nowhere dense. Otherwise, we would have uncountable many disjoint
open sets in $K_0$. 
Take $h:[0,1]\longrightarrow [0,1]$ a continuous, bijective and increasing
function such that $h(r)=\frac{1}{2}$. Replace $e_n$ by $h\circ e_n$. 
Repeat the process for every $n\in\N$. 

\begin{clm}\label{clm:dense} $D((e_n)_{n\in\N})=K_0\smallsetminus\{\bar{x}\}$.
\end{clm}

To prove the claim we note that for every open neighborhood $U$ of $\bar{x}$,
there exists a finite $F\subset \N$ such that $U_n\subset U$, 
for every $n\in N\smallsetminus F$. Since $supp(e_n)\subset U_n$ and it is non-empty,
we have $\bar{x}\in K_0\smallsetminus D((e_n)_{n\in\N})$. On the other hand, 
if $x\neq \bar{x}$, by Hausdorff there exist disjoint open sets $V$ and $U$
such that $x\in V$ and $\bar{x}\in U$. We have $U_n\subset U$ for all but
finitely many $n\in\N$, concluding that $x\in D((e_n)_{n\in\N})$ and proving the claim.

\begin{clm}\label{clm:pi} $\pi_{K_1,K_0}^{-1}[\bar{x}]=\{\bar{x}\}\times [0,1]$.
\end{clm}

Let $t\in [0,1]$. Using that every $e_n$ is surjective on $[0,1]$,
take $y_n\in U_n$ such that $e_n(y_n)=t$. Clearly, $(y_n)_{n\in\N}$ converges to $\bar{x}$
and every pair $(y_n,t_n)$ belongs to the graph of $\Sigma_{n\in\N} e_n|D((e_n)_{n\in\N})$,
which is included in  $K_1$.
Hence, by compactness of $K_1$, we have $(\bar{x},t)\in K_1$, proving the claim. 

\vspace{2mm}
  
Let $(a_n)_{n\in\N}$ and $(b_n)_{n\in\N}$ be sequences in $[0,1]$ both
converging to $\frac{1}{2}$ such that $a_{n+1}<b_n<a_n<1$, for every $n\in\N$. 
Define $\tilde{f}_n:[0,1]\longrightarrow [0,1]$ with support included in $[a_{n+1},a_n]$
and such that $\tilde{f}_n(b_n)=1$. Suppose also that $\tilde{f}_n$ is monotone in each interval
$[a_{n+1},b_n]$ and $[b_n,a_n]$, as in the definition of $g_n$ in the proof
of \ref{thm:disconnected1}. 

Let $f_n(x,t)=\tilde{f}_n(t)$, for all $(x,t)\in K_1\subset K_0\times [0,1]$.
Define $K_2=K_1((f_n)_{n\in\N})$.  
 
\begin{clm}\label{clm:dense1}
 $K_1\smallsetminus (K_0\times\{\frac{1}{2}\})\subset D((f_n)_{n\in\N})$. 
\end{clm}

The claim follows from the fact that $D((\tilde{f}_n)_{n\in\N})=[0,1]\smallsetminus\{\frac{1}{2}\}$.
If we take $t\neq\frac{1}{2}$ and $V$ an open neighborhood of $t$ which intersects 
$supp(\tilde{f}_n)$ for finitely many $n$'s, the set $(K_0\times V)\cap K_1$ intersects $supp(f_n)$
for finitely many $n$'s, since $f_n(x,t)=\tilde{f}_n(t)$.

\vspace{2mm}

Let $K'$ be the extension of $[0,1]$ by $(\tilde{f}_n)_{n\in\N}$. 

Let $\tilde{g}_n:K'\longrightarrow [0,1]$ defined as $f_n$ in the proof of 
theorem~\ref{thm:disconnected2}. So we have $K'((\tilde{g}_n)_{n\in\N})$ disconnected.
Call it $K''$. Define $$g_n(x,y,z)=\tilde{g}_n(y,z)$$ on $K_2$.

\begin{clm}\label{clm:welldefined} Every $g_n$ is well defined and continuous, 
and $(g_n)_{n\in\N}$ is pairwise disjoint.
\end{clm}

Continuity and disjointness follow immediately from the definition, since the projection
is continuous. To prove that the functions are well defined, we have to show that $(y,z)\in K'$,
wherever $(x,y,z)\in K_2$.
If $y\neq\frac{1}{2}$, it follows from claim~\ref{clm:dense1} that $z=f_n(x,y)$ and, therefore,
$z=\tilde{f}_n(y)$. So we have $(y,z)\in K'$. If $y=\frac{1}{2}$, by the construction we have 
$(y,z)\in K'$ for any $z\in [0,1]$, proving the claim.

\vspace{2mm}

Now we take $K_3=K_2((g_n)_{n\in\N})$.

\begin{clm}\label{clm:nowheredense} 
$K_3\cap (K_0\times \{\frac{1}{2}\}\times [0,1]^2)$ is nowhere dense in
$K_3\cap (K_0\times [\frac{1}{2},1]\times [0,1]^2)$.
\end{clm}

By claim~\ref{clm:nowheredense1}, $K_1\cap (K_0\times \{\frac{1}{2}\})$ is nowhere dense in $K_1$.
Using lemma~\ref{lem:restriction}, $K_2\cap K_0\times [\frac{1}{2},1]\times [0,1]$ is
the extension of $K_1\cap (K_0\times [\frac{1}{2},1])$ by the restrictions of $(f_n)_{n\in\N}$, 
regarding that $f_n(x,t)=0$, if $t<\frac{1}{2}$. Therefore, since extensions preserve nowhere dense 
sets in the inverse projection (see~\cite{Ko1}), we have $K_2\cap K_0\times\{\frac{1}{2}\}\times [0,1]$
nowhere dense in $K_2\cap (K_0\times [\frac{1}{2},1]\times [0,1])$. Repeating these arguments
in the next extension, we conclude the proof of the claim.

\vspace{2mm}

Let $L=K_3\cap (K_0\times [0,\frac{1}{2}]\times\{(0,0)\})$. Clearly $L$ is a closed
subspace of $K_3$. To conclude the disconnectedness of $K_3$, it is enough to prove the
following claim:

\begin{clm}\label{clm:disconnected} $\overline{K_3\smallsetminus L}\cap L=\emptyset$.
\end{clm}

Suppose, by absurd, that there exists a sequence $(x_n,t_n,s_n,r_n)$ in $K_3\smallsetminus L$ 
converging to $(x,t,s,r)\in L$. By metrizability and claim~\ref{clm:nowheredense} we may
assume that $t_n>\frac{1}{2}$, for every $n$. Therefore, as we proved, $r_n=\tilde{g}_m(t_n,s_n)$,
for some $m$. So, we have $(t_n,s_n,r_n)\in K''$. Clearly we have $(t,s,r)=(\frac{1}{2},0,0)$.
This contradicts the final argument in the proof of theorem~\ref{thm:disconnected2},
when we proved the disconnectedness of $K''$.

\end{proof}

\section{Final remarks}

Extensions by continuous functions on connected spaces have been useful to study
the geometry of Banach spaces of the form $C(K)$. For instance, the technique of extensions was introduced
in~\cite{Ko1} to construct the first indecomposable $C(K)$, which is also
the first $C(K)$ non-isomorphic to any $C(L)$, for $L$ totally disconnected.
Preservation of connectedness in such construction is essential, and requires some
additional properties on the extensions, as the definition of \emph{strong extensions}.
We show in this paper that such kind of requirement is necessary, since preservation of
connectedness may fail.

Theorem~\ref{thm:locally} proves that extensions preserve connectedness when $K$ is locally
connected. In particular, a single extension of the interval $[0,1]$ must be connected. 
Nevertheless, the extension usually loses the property of locally connectedness.

Theorem~\ref{thm:disconnected1} provides a three-dimensional visual example which shows the
failure of preservation of connectedness. Theorem~\ref{thm:disconnected2} shows how this
example can be adapted as a double extension of the interval $[0,1]$. Finally,
theorem~\ref{thm:disconnected3} adapts the proof to higher dimensions,
proving that from any metrizable connected compactum we can get a disconnected space after
three successive extensions.

We may rephrase theorem~\ref{thm:disconnected3} as the following: there is not a 
non-empty class $\mathcal{C}$ of continua such that, whenever $K\in\mathcal{C}$ and $L$ is  
an extension of $K$ by continuous functions, then $L\in\mathcal{C}$. 

One further question to be investigated is: what happens if we take off the hypothesis of
metrizability? Is there a non-empty class $\mathcal{C}$ of connected compacta which is closed
by taking extensions by continuous functions? Although this question is interesting itself,
even a positive answer to it probably would not help in constructions of $C(K)$, since
most of these constructions use induction which starts with a metrizable compactum. 

Furtherer, we may still looking for others conditions on $K$ and $(f_n)_{n\in\N}$
that imply the preservation of connectedness, and how this impacts on the theory of
Banach spaces of the form $C(K)$. We also may look for others ways of adding suprema
of continuous functions on connected spaces -- as the one made in~\cite{Ko3} --
and its applications in functional analysis.

\end{document}